\documentclass[oneside,english]{amsart}
\usepackage[T1]{fontenc}
\usepackage[latin9]{inputenc}
\usepackage{babel}
\usepackage{units}
\usepackage{textcomp}
\usepackage{amsthm}
\usepackage{amstext}
\usepackage{amssymb}
\usepackage{setspace}
\usepackage[unicode=true,
 bookmarks=true,bookmarksnumbered=false,bookmarksopen=false,
 breaklinks=false,pdfborder={0 0 1},backref=false,colorlinks=false]
 {hyperref}
\hypersetup{pdftitle={Rank Constrained Homotopies},
 pdfauthor={Kaushika De Silva and Andrew Toms}}
\usepackage{breakurl}

\makeatletter
\numberwithin{equation}{section}
\numberwithin{figure}{section}
\theoremstyle{plain}
\newtheorem{thm}{\protect\theoremname}[section]
  \theoremstyle{definition}
  \newtheorem{defn}[thm]{\protect\definitionname}
  \theoremstyle{remark}
  \newtheorem{rem}[thm]{\protect\remarkname}
  \theoremstyle{plain}
  \newtheorem{cor}[thm]{\protect\corollaryname}
  \theoremstyle{remark}
  \newtheorem*{rem*}{\protect\remarkname}
  \theoremstyle{plain}
  \newtheorem{prop}[thm]{\protect\propositionname}
  \theoremstyle{plain}
  \newtheorem{lem}[thm]{\protect\lemmaname}
  \renewcommand\@biblabel[1]{#1.}

\makeatother

  \providecommand{\corollaryname}{Corollary}
  \providecommand{\definitionname}{Definition}
  \providecommand{\lemmaname}{Lemma}
  \providecommand{\propositionname}{Proposition}
  \providecommand{\remarkname}{Remark}
\providecommand{\theoremname}{Theorem}

\begin{document}
\address{Department of Mathematics \\ Purdue University \\ 150 N. University St. \\ West Lafayette IN \\ USA \\ 47907 }
\email{pdesilva@math.purdue.edu}

\begin{onehalfspace}

\title{\noindent Rank Constrained Homotopies}
\end{onehalfspace}

\author{\noindent Kaushika De Silva}
\begin{abstract}
\noindent For any $n\geq k\geq l\in\mathbb{N},$ let $S(n,k,l)$ be
the set of all those non-negative definite matrices $a\in M_{n}(\mathbb{C})$
with $l\leq\text{rank }a\leq k$. Motivated by applications to $C^{*}$-algebra
theory, we investigate the homotopy properties of continuous maps
from a compact Hausdorff space $X$ into sets of the form $S(n,k,l).$
It is known that for any $n,$ if $k-l$ is approximately
4 times the covering dimension of $X$ then there is only one homotopy
class of maps from $X$ into $S(n,k,l)$, i.e. $C(X,S(n,k,l))$ is
path connected. In our main Theorem we improve this bound by a factor
of 8. By combining classical homotopy theory methods with $C^{*}$-algebraic
techniques we also show that if $\pi_{r}(S(n,k,l))$ vanishes for
all $r\leq d$ then $C(X,S(n,k,l))$ is path connected for any compact Hausdorff $X$
with covering dimension not greater than $d$. 
\end{abstract}

\maketitle

\section{\noindent Introduction.}

\noindent For $n\in\mathbb{N}$, let $(M_{n})_{+}$ stand for the
$n\times n$ non-negative definite matrices over $\mathbb{C}$ and
for fixed $k,l\in\mathbb{N}$ with $n\geq k\geq l$ define the space
$S(n,k,l)$ by,

\noindent 
\[
S(n,k,l)=\left\{ b\in(M_{n})_{+}|\,\ensuremath{l\leq}\text{rank}(\ensuremath{b)\leq k}\right\} .
\]

\noindent Let $S(n,k,l)$ be given the subspace topology induced by
the norm topology of $M_{n}$. We focus on homotopy properties of
$S(n,k,l)$. Our main theorem is the following;

\noindent \textbf{Theorem 3.5}. \textit{Let $X$ be a compact Hausdorff
space $X$ with $\left\lfloor \frac{dim\, X}{2}\right\rfloor \leq k-l.$
Then there is only one homotopy class of functions $f:X\to S(n,k,l),$
i.e. the function space $C(X,S(n,k,l))$ is path connected.}

\noindent Here, by $dim\: X$ we mean the covering dimension of the
space $X$ and $\left\lfloor m\right\rfloor $
stands for the largest integer which is less than or equal to $m\in\mathbb{R}.$
An immediate consequence of Theorem 3.5 is that $\pi_{r}(S(n,k,l))=0, \forall r\leq2(k-l)+1$.

\noindent The spaces $S(n,k,l)$, or rather their homotopy properties
have relevance in $C^{*}$-algebra theory and our motivation for this
study is mainly due to this natural connection. 

\noindent For a unital $C^{*}$-algebra $A$, let $T(A)$ denote the
tracial state space. In other words $T(A)$ is the set of all normalized positive linear functionals $\tau$ on $A$ which satisfy the trace property $\tau(ab)=\tau(ba),\forall a,b\in A$. An important object of study for a $C^{*}$-algebraist
is the space $LAff_{b}(T(A))^{++}$ of all bounded, strictly positive,
lower semi-continuous affine maps on $T(A)$. For a positive element
$a$ in $A$, $\tau\mapsto\iota_{a}(\tau)=\lim_{n}\tau(a^{\nicefrac{1}{n}})$
defines a positive lower semi-continuous affine map on $T(A)$. This
definition extends to positive elements in $M_{n}(A),n\geq2$, in
which case $\tau$ in the right hand side stands for the non-normalized
canonical extension of $\tau$ to $M_{n}(A)$. If $A$ is simple $\iota_{a}$
is strictly positive for all $a\in M_{n}(A)$ for any $n\in\mathbb{N}$.
For sufficiently regular simple $C^{*}$-algebras (as in \cite{BPT},\cite{tm1}),
$\lbrace\iota_{a}\colon a\in M_{\infty}(A)_{+}\rbrace$ is dense in
$LAff_{b}(T(A))^{++}$. On the other hand, density of $\lbrace\iota_{a}\colon a\in M_{\infty}(A)_{+}\rbrace\subseteq LAff(T(A))^{++}$
provides one with useful tools to work with to derive classification
results and to establish regularity properties such as $\mathcal{Z}$-stability
\cite{BT,tm1}. In \cite[Theorem 3.4]{tm1}, Toms shows that for unital
simple ASH algebras with slow dimension growth $\lbrace\iota_{a}\colon a\in M_{\infty}(A)\rbrace$
is dense in $LAff_{b}(T(A))^{++}$. He then uses this fact to show that
such algebras are $\mathcal{Z}$-stable and hence derives classification
results. The proof of \cite[Theorem 3.4]{tm1} heavily depends on
homotopy properties of the spaces $S(n,k,l)$ \cite[Section 2]{tm1}.
Our main result here is an improvement of \cite[Proposition 2.5]{tm1},
which is a key ingredient in proving that $\lbrace\iota_{a}\colon a\in M_{\infty}(A)_{+}\rbrace\subseteq LAff(T(A))^{++}$
is dense for the class of algebras mentioned above. However, one should
note that this improvement does not have a direct impact on the main
result of \cite{tm1}.

\subsection*{Spaces $S(n,k,l)$ as a generalization of complex Grammarians}

\

\noindent Let us consider the special case $k=l.$ Recall that $G_{k}(\mathbb{C}^{n})$
stands for the complex Grassmann variety of $k$-dimensional subspaces
of $\mathbb{C}^{n}.$ Given a $k$ -dimensional subspace $V$ of $\mathbb{C}^{n},$
one may identify it uniquely with the orthogonal projection of $\mathbb{C}^{n}$
on $V$. This identification leads to a natural homeomorphism from
$G_{k}(\mathbb{C}^{n})$ to $P_{k}(\mathbb{C}^{n})$, where $P_{k}(\mathbb{C}^{n})$
is the set of all rank $k$ projections in $M_{n}(\mathbb{C})$ equipped
with the norm topology. It is not hard to see that the inclusion $P_{k}(\mathbb{C}^{n})\subset S(n,k,k)$
induces a homotopy equivalence. Hence, $G_{k}(\mathbb{C}^{n})$ is
homotopy equivalent to $S(n,k,k).$ Thus in a sense, the spaces $S(n,k,l)$
can be viewed as a generalization of the Grassmann varieties, at least
for homotopy interests. By setting $k=l$ in Theorem 3.5 and using
the homotopy equivalence of $S(n,k,k)$ with $G_{k}(\mathbb{C}^{n})$
one can derive that $G_{k}(\mathbb{C}^{n})$ is simply connected (i.e.
$\pi_{1}(G_{k}(\mathbb{C}^{n}))=0$) for any pair of $k,n$, which
is a well known classical result. However, it should be noted that
our proof of Theorem 3.5 does not in provide a alternate proof of
this fact, rather we use a stronger classical result in our proof.

\subsection*{Rank varying bundles associated with the spaces $S(n,k,l)$}

\

Given a continuous map $f:X\to G_{k}(\mathbb{C}^{n})$, one has the
associated locally trivial vector bundle $f^{*}(\gamma)$ over $X$,
which is the pullback of the canonical $k$-dimensional vector bundle
$\gamma$ over $G_{k}(\mathbb{C}^{n})$ to $X$.

\noindent In a similar vein, for each map $a\in C(X,(M_{n})_{+})$
there is a naturally associated bundle (in the sense of \cite[Chapter 2]{Huse})
$\epsilon_{a}$ over $X$ with the total space, 
\[
E_{a}=\left\{ (x,v)\in X\times\mathbb{C}^{n}|v\in a(x)(\mathbb{C}^{n})\right\} 
\]
and $\pi_{1}:E_{a}\to X$ being the natural coordinate projection. 
Admittedly, a typical bundle obtained in this nature is not necessarily
a vector bundle in the classical sense (\cite{At,Huse}). In fact
such a bundle may not even have a constant fiber. Still, for a bundle
$\epsilon_{a}$ each fiber admits a natural vector space structure,
which to an extent preserves local triviality. Moreover these bundles
occur naturally in $C^{*}$-algebra theory (see \cite{tm2},\cite{tm3})
and understanding this association would make the techniques used
in section 3 more intuitive. Thus, before moving on to the next section
we outline the structure of these bundles.

\subsection*{The structure of bundles associated with the spaces $S(n,k,l)$}

\

\noindent Observe that if $a(x)\in S(n,k,k),\forall x\in X$, then
the bundle $\epsilon_{a}$ is indeed a locally trivial vector bundle
over $X$. However, if $a$ is not of constant rank the associated
bundle $\epsilon_{a}$ is not locally trivial in the usual sense.
Still, if we set $E_{i}=\left\{ x\in X:\text{ rank}(a(x))=n_{i}\right\} $,
then continuity of $a$ implies that the support projection of $a$
is continuous on $E_{i}$. Here, by support projection of $a$ we
mean the function which assigns each $x$ to the orthogonal projection
on the subspace $a(x)(\mathbb{C}^{n})$. Hence, the restriction of
$\epsilon_{a}$ to $E_{i}$ - denoted by $\epsilon_{a}\restriction_{E_{i}}$,
is a locally trivial vector bundle over $E_{i}$. In this manner we
can partition $X$ into a finite collection of subsets, such that
the restriction of $\epsilon_{a}$ to each subset is a locally trivial
bundle of constant rank. One may now consider the possibility of applying classical vector bundle theory \cite{At,Huse} to establish the structure of $\epsilon_a$ in some local sense, i.e the structure of $\epsilon_{a}\restriction_{E_{i}}$ for each $i$. But the issue with this is that the subsets $E_{i}$ formed here are highly non regular where as to apply classical theory of vector bundles the base spaces need to satisfy regularity properties such as compactness. We can overcome this
if $q_{i}$- the support projection of $a$ on $E_{i}$, extends to
a projection $p_{i}\in M_{n}(C(\overline{E}_{i}))$ for each $i$,
where $\overline{E}_{i}$ is the closure $E_{i}$. If this is the
case, then we can apply classical vector bundle theory (see \cite{At,Huse})
to understand the structure of bundles $\epsilon_{p_{i}}$. Moreover,
if its also the case that for any two distinct $i,j$ with $\overline{E}_{i}\cap\overline{E}_{j}\neq\emptyset$
the extended projections $p_{i}$ and $p_{j}$ are \textsl{comparable}
on $\overline{E}_{i}\cap\overline{E}_{j}$, then we can expect to
use structural properties of the bundles $\epsilon_{p_{i}}$ (i.e
local structure of $\epsilon_{a}$) to establish the global structure
of $\epsilon_{a}$. 

\noindent The above considerations would not hold true for arbitrary $a\in M_{n}(C(X))_{+}.$
However, {[}\cite{tm3}, c.f. \cite{Phil}{]} introduces a special
class of elements in $M_{n}(C(X))_{+}$ called\textsl{ well supported
}positive elements, which are well behaved in the above sense. For
an element $a$ in this class (Definition 2.2), each $q_{i}$ extends
to $p_{i}\in M_{n}(C(\overline{E}_{i}))$ in such a way that for $i<j$
with $\overline{E}_{i}\cap\overline{E}_{j}\neq\emptyset$, $p_{i}$
is a sub-projection of $p_{j}$ of on $\overline{E}_{i}\cap\overline{E}_{j}$.
From results in \cite{tm2} it follows that for any $n,k,l\in\mathbb{N}$
the set of all well supported positive elements contained in $C(X,S(n,k,l))$
is homotopy equivalent to $C(X,S(n,k,l))$ {[}Lemma 3.1{]}. This fact
plays a crucial role in the proof of our main theorem. 

\noindent In section 2, we briefly recall some well known definitions
and results from vector bundle theory \cite{At,Huse} as well as some
other tools and notations that are necessary. In section 3, following
the ideas from \cite{Phil} and \cite{tm1} we prove the main result.
Section 4 (Theorem 4.6) shows that for any $n,k,l\in\mathbb{N}$ the
path connectedness of $C(X,S(n,k,l))$ depends solely on homotopy
of the space $S(n,k,l).$ In section 4 we do not assume that $k-l\geq\lfloor\frac{dim\, X}{2}\rfloor$
and the proof of Theorem 4.6 is achieved mainly through applications
of classical homotopy theory results (see \cite{Wh}), to the space
$S(n,k,l).$ 

\noindent \textbf{Acknowledgments}\textsl{.} I thank my adviser Professor
Andrew Toms for his encouragement, for several helpful conversations
and suggestions that were invaluable. This work wouldn't have been
possible without his efforts.  I am also thankful to the reviewer of
an earlier version of this article, for providing thoughtful suggestions,
comments and references.

\section{\noindent Notations and Preliminaries.}

\subsection*{Notations and Conventions.}

\noindent Unless stated otherwise we assume $X$ to be a compact Hausdorff
space. $C(X,Y),\, M_{n\times m}(C(X)),\, M_{n}(C(X))$ all have the
usual meanings and $C(X)=C(X,\mathbb{C})$. We will often identify
$M_{n}(C(X))$ with $C(X,M_{n}),$ where $M_{n}$ denotes $M_{n}(\mathbb{C})$.
$(M_{n})_{+}$ denotes all non negative definite matrices in $M_{n}$
and its customary to call these as positive elements of $M_{n}.$

\noindent Given $x\in M_{n}$, $x^{*}$ denotes the conjugate transpose
of $x$. By a projection $p$ in $M_{n}$ we mean a self adjoint idempotent,
i.e $p=p^{2}=p^{*}.$ Note that the notions of the conjugate of an
element, positive elements and projections carry over to $M_{n}(C(X))$
via point wise defined operations. We will use $M_{n}(C(X))_{+}$
to denote positive elements in $M_{n}(C(X))$.

The base field for all vector spaces and vector bundles that we will
consider is the field of complex numbers.

\subsection*{Murray-von Neumann semi group of $C(X)$, semi group of isomorphism
classes of vector bundles over $X$ and Serre-Swan Theorem.}

\

\noindent Let $P_{\infty}(C(X))=\underset{m\in\mathbb{N}}{\bigcup}P(M_{m}(C(X)))$,
where $P(M_{m}(C(X)))$ denote the set of all projections in $M_{m}(C(X))$.
A pair $p\in P(M_{m}(C(X)))$ and $q\in P(M_{n}(C(X)))$ are said
to be Murray-von Neumann equivalent, denoted $p\sim q$, if there
is some \mbox{$v\in M_{n,m}(C(X))$} with $p=v^{*}v$ and $q=vv^{*}$.
The Murray-von Neumann equivalence class of $p\in P_{\infty}((C(X)))$
is denoted by $\left[p\right]_{0}$ and the Murray-von Neumann semigroup
of $C(X)$ is the semigroup 
\[
D(C(X))=\left(P_{\infty}(C(X))/\sim,+\right),
\]
where $\left[p\right]_{0}+\left[q\right]_{0}=\left[p\oplus q\right]_{0}$. 

\noindent Let $\text{Bun}(X)$ denote the set of locally trivial complex
vector bundles over $X$, and let $\text{Bun}_{k}(X)$ denote the
set of all $\text{\ensuremath{\epsilon\in}Bun}(X)$ of constant fiber
dimension $k$. Write $\text{Vect}(X)$ to denote the semigroup of
all isomorphism classes of bundles in $\text{Bun}(X)$ with addition
being induced via the direct sum of bundles. The well known natural
identification of $D(C(X))$ with Vect($X$), induced through the
Serre-Swan Theorem \cite[Theorem 2]{Swn} (which states that $\text{Bun}(X)$
as a category with morphisms being the morphisms of vector bundles
is equivalent to the category of finitely generated projective $C(X)$-modules),
is central to our study. For the sake of completeness and to introduce
some of the notations and terms that we will use, let us describe
this identification. 

\noindent For $p\in P(M_{m}(C(X)))\subset M_{n}(C(X))_{+}$, let $\epsilon_{p}=(E_{p},\pi_{1},X)$
be the bundle over $X$ defined as in the introduction. Since $p$
is a continuous projection, $\epsilon_{p}$ is a locally trivial vector
bundle over $X$. Moreover, $p\sim q$ for some $q\in P_{\infty}(C(X))$,
if and only if \mbox{$\mathcal{\epsilon}_{p}\cong\mathcal{\epsilon}_{q}$}.
This gives a well defined injection $\psi:D(C(X))\to\text{Vect}(X)$,
which is a semigroup morphism and preserves dimensions. On the other
hand, suppose $\epsilon$ is a locally trivial vector bundle over
$X$, with total space $E$ and the fiber at $x\in X$ being $E_{x.}$
From \cite[Corollary 5]{Swn}, $\epsilon$ is a direct summand of
a trivial bundle. That is, there is some $n\in\mathbb{N}$ and a bundle
$\epsilon^{\perp}$ over $X$ such that $\epsilon\oplus\epsilon^{\perp}\cong\theta^{n}$,
where $\theta^{n}$ is the $n$ dimensional product bundle. Bundle
$\epsilon^{\perp}$ is called a complementary bundle for $\epsilon.$
Let $F$ denote the total space of $\epsilon^{\perp}$and let $F_{x}$
denote the fiber at $x\in X.$ Then, we may assume that $E_{x}\oplus F_{x}=\mathbb{C}^{n},\forall x\in X$.
Let $p_{\epsilon}(x)$ denote the orthogonal projection of $\mathbb{C}^{n}$
on the fiber $E_{x}\subset\mathbb{C}^{n}$. Then, the local triviality
of $\epsilon$ ensures that the map $x\mapsto p_{\epsilon}(x)$ is
continuous. Indeed, $\psi(p_{\epsilon})=\epsilon$ and thus $\psi$
is a surjection.

\noindent Recall that a vector bundle $\eta$ over $X$ is called
trivial if $\eta$ is isomorphic to a product bundle over $X$. We
say that a projection $p$ in $M_{n}(C(X))$ is a \textsl{trivial}
projection if $\epsilon_{p}$ is trivial, where $\epsilon_{p}$ is
defined as before. From the preceding paragraph it is clear that if
$p,q\in M_{n}(C(X))$ with $p\sim q$, then $p$ is trivial iff $q$
is trivial.

\subsection*{Bounding the dimension of $\epsilon^{\perp}$.}

\

From the preceding for each $\epsilon\in\text{Bun}(C(X))$ there is
$\epsilon^{\perp}$ so that $\epsilon\oplus\epsilon^{\perp}\cong\theta^{n}$
for some $n\in\mathbb{N}.$ We would like to have a bound on the dimension
on $\epsilon^{\perp}$. Having such a bound will be useful for our
work in section 3 as there we focus on $S(n,k,l)$ for a fixed triple
$(n,k,l)$. Let us now address this concern and provide a bound for
the dimension of $\epsilon^{\perp}$ which depends only on the dimension
of the base space $X.$

Recall the following Theorem on locally trivial vector bundles over
a compact Hausdorff space. Here, $\lceil x\rceil$ stands for the
smallest integer $n$ with $x\leq n$.
\begin{thm}
\noindent Let $X$ be a finite dimensional compact Hausdorff space.
Let $\epsilon,\gamma\in\text{Bun}(X)$ and $\theta^{n}$ denote the
product bundle of dimension $n$. Let $k$ be the dimension of $\epsilon$
and write $m=\lfloor\frac{dim\, X}{2}\rfloor$. The following hold,

\noindent \textup{1.} $\epsilon\cong\eta\oplus\theta^{k-m}$ for some
$\eta\in\text{Bun}_{m}(X)$.

\noindent \textup{2}. If $k\geq\left\lceil \frac{dim\, X}{2}\right\rceil $
and $\epsilon\oplus\delta\cong\gamma\oplus\delta$, where $\delta$
is a trivial bundle over $X,$ then $\epsilon\cong\gamma$. 
\end{thm}
\noindent A proof of Theorem 2.1 in the case of $X$ being a $CW$-complex is
given by Husemoller \cite[Chap. 9, Theorems 1.2 and 1.5]{Huse}. Any
compact Hausdorff space $X$ of dimension $d$ can be realized as
the inverse limit of compact metric spaces $X_{\alpha}$, with $\text{dim }X_{\alpha}\leq d$
for each $\alpha$ \cite[Chap 27, Theorem 8]{Nagami}
and any compact metric space $Y$ is homeomorphic to an inverse limit
of finite simplicial complexes of dimension not greater than that
of $Y$ \cite[Chap. 27, Theorem 12]{Nagami}). In
\cite[Theorem 2.5]{Goodreal}, Goodearl uses these identifications and Husmoller's results \cite[Chap. 9, Theorems 1.2 and 1.5]{Huse} to prove the Theorem for arbitrary compact Hausdorff spaces.

\noindent One immediate consequence of Theorem 2.1 is that, given
$\epsilon\in\text{Bun}_{k}(X)$ we may choose a complementary bundle
$\epsilon^{\perp}$ so that the dimension of $\epsilon^{\perp}$ is
at most $\left\lfloor \frac{dim\, X}{2}\right\rfloor $.

\noindent To observe this, first choose some $\delta\in\text{Bun}(X)$,
say of dimension $t$, so that $\epsilon\oplus\delta\cong\theta^{k+t}$.
If $t>\left\lfloor \frac{dim\, X}{2}\right\rfloor $, by part 1 of
the Theorem, $\delta=\gamma\oplus\theta^{t-\left\lfloor \frac{dim\, X}{2}\right\rfloor }$
for some $\gamma\in\text{Bun}_{\left\lfloor \frac{dim\, X}{2}\right\rfloor }(X)$.
If dimension of $\epsilon\oplus\gamma$ is $k_{1}$, then $k_{1}=k+\left\lfloor \frac{dim\, X}{2}\right\rfloor \geq\left\lceil \frac{dim\, X}{2}\right\rceil $
and $(\epsilon\oplus\gamma)\oplus\theta^{t-\left\lfloor \frac{dim\, X}{2}\right\rfloor }\cong\theta^{k+t}$.
Hence by part 2 of Theorem 2.1, it follows that $\epsilon\oplus\gamma\cong\theta^{k+\left\lfloor \frac{dim\, X}{2}\right\rfloor }$.
Thus for each $\epsilon\in\text{Bun}_{k}(X)$ we may choose $\epsilon^{\perp}$
to be of dimension $k+\left\lfloor \frac{dim\, X}{2}\right\rfloor .$ 

\noindent Upshot of all this is that for fixed $n,k\in\mathbb{N}$
with $n\geq k+\left\lfloor \frac{dim\, X}{2}\right\rfloor $, by following
the methods discussed in the preceding subsection we can now construct
a bijective correspondence between the isomorphism classes of $\text{\textbf{\text{Bun}}}_{k}(X)$
and Murray-von Neumann equivalence classes of projections in $P_{k}(M_{n}(X))$.
For convenience let us call this the \textsl{Serre-Swan correspondence.}
In section 3, we will combine this correspondence with part 1 of Theorem
2.1 to construct continuous maps \mbox{$a:X\to(M_{n})_{+}$}, which
satisfy specified rank constrains.

\subsection*{Well supported positive elements.}

\

\noindent Given $a\in M_{n}(C(X))_{+}$, the rank function of $a$
- denoted $r_{a}$, is the (lower semi-continuous) function defined
on $X$ by $x\mapsto\text{rank}(a(x))$. For $a,b\in M_{n}(C(X))$
if $a-b\in M_{n}(C(X))_{+},$ we write $a\leq b.$

\noindent We now give the definition of a\textsl{ well supported }positive
element in $M_{n}(C(X))$, given in \cite{tm3}. 
\begin{defn}
\noindent Let $X$ be a compact Hausdorff space and let $a\in M_{n}(C(X))_{+}$.
Suppose that $n_{1}<n_{2}<\cdot\cdot\cdot<n_{L}$ denote all the values
that $r_{a}$ takes and set $E_{i}=\left\{ x\in X\,:\text{}r_{a}(x)=n_{i}\right\} $.
We say that $a$ is \textsl{well supported} if, for each $1\leq i\leq L$
there is a projection $p_{i}\in M_{n}(C(\overline{E}_{i}))$ such
that $\lim_{r\to\infty}a(x)^{1/r}=p_{i}(x),\forall x\in E_{i}$, and
$p_{i}(x)\leq p_{j}(x)$ whenever $x\in\overline{E}_{i}\cap\overline{E}_{j}$,
and $i\le j$.
\end{defn}

\noindent Following Theorem can be used to replace arbitrary positive
elements by well supported ones up to homotopy (see Lemma 3.1).

\begin{thm}
\noindent \cite[Theorem 3.9]{tm2}. Let X be a compact Hausdorff space
and let \mbox{$a\in M_{n}(C(X))_{+}$}. Then, for every $\delta>0$,
there exists a well supported element $b\in M_{n}(C(X))_{+}$ such
that $b\leq a$ and $||a-b||<\delta$, with the range of $r_{b}$
equal to the range of $r_{a}$.\end{thm}
\begin{rem}
\noindent The Theorem stated above is only a part of Theorem 3.9 of
\cite{tm2}. There, in the hypothesis $X$ is assumed to be a finite
simplicial complex. However, the simplicial structure of $X$ is required
only for the second part of the Theorem, which guarantees that each
$\overline{E}_{i}$ corresponding to $b$ (as defined in 2.2) can
assumed to be a sub-complex of $X.$ For the Theorem as stated here,
$X$ being compact and Hausdorff suffices. 
\end{rem}

\subsection*{Some useful extension results.}

\

\noindent Given a function $f$ on $X$ and a $Z\subset X$, we use
$f\restriction_{Z}$ to denote the restriction of $f$ to $Z$.

\noindent Let $a\in M_{n}(C(X))_{+}$ be well supported and fix a
closed subset $Y\subset X$. Let $q$ be a trivial projection on $Y$,
majorized by the support projection of $a\restriction_{Y}$. In section
3 we would require to extend $q$ to a trivial projection defined
on $X$, such that extension is also majorized by the support projection
of $a$. To accomplish this, we follow ideas developed by  Toms \cite{tm3}
based on the work of Phillips \cite{Phil}. 
\begin{thm}
\noindent \cite[Proposition 4.2 (1)]{Phil}. Let X be a compact Hausdorff
space of dimension $d<\infty$, and let $Y\subset X$ be closed. Let
$p,q\in M_{n}(C(X))$ be projections with the property that $\text{rank}(q(x))+\left\lfloor \frac{d}{2}\right\rfloor \leq\text{rank}(p(x))$,
$\forall x\in X$. Let $s_{0}\in M_{n}(C(Y))$ be such that $s_{0}^{\ast}s_{0}=q\restriction_{Y}$
and $s_{0}s_{0}^{\ast}\leq p\restriction_{Y}$. It follows that there
is $s\in M_{n}(C(X))$ such that $s^{\ast}s=q,\, ss^{\ast}\leq p$,
and $s_{0}=s\restriction_{Y}$.\end{thm}

\begin{cor}
\noindent \cite[corollary 2.7. (ii)]{tm3}. Let $X$ be a compact
Hausdorff space of dimension $d<\infty$, and let $E_{1},...,E_{k}$
be a cover of $X$ by closed sets. Let $q\in P(M_{n}(C(X)))$ and
for each $i\in{1,...,k}$ let \mbox{$p_{i}\in M_{n}(C(E_{i}))$} be
a projection of constant rank $n_{i}$. Assume that $n_{1}<n_{2}<\cdot\cdot\cdot<n{}_{k}$
and $p_{i}(x)\leq p_{j}(x)$ whenever $x\in E_{i}\cap E_{j}$ and
$i\leq j$. Finally, suppose that $n_{i}-\text{rank }(q)\geq\left\lfloor \frac{d}{2}\right\rfloor $
for every $i.$ Then the following hold:

\noindent If $Y\subset X$ is closed, $q\restriction_{Y}$ is trivial
and $\forall y\in Y,$ 
\[
q(y)\leq\underset{\left\{ i|y\in E_{i}\right\} }{\bigwedge}p_{i}(y),
\]
 then $q\restriction_{Y}$ can be extended to a projection $\widetilde{q}$
on $X$ which is also trivial and satisfies,
\end{cor}
\noindent 
\[
\widetilde{q}(x)\text{\ensuremath{\leq}}\underset{\left\{ i|x\in E_{i}\right\} }{\bigwedge}p_{i}(x),\forall x\in X.
\]

\noindent The following probably is a well known result, but we could
not find a direct reference and thus we include a short proof by applying
Theorem 2.5. 

\begin{cor}
\noindent Let $X$ be a compact Hausdorff space of dimension $d<\infty$,
and let $Y\subset X$ be closed. Then any trivial projection $r\in M_{n}(C(Y))$
extends to a trivial projection on $X$, provided that $\text{rank }(r)\leq n-\left\lfloor \frac{d}{2}\right\rfloor .$\end{cor}
\begin{proof}

\noindent Suppose that the rank of $r$ is $k$ and take $q$ to be
any projection in $M_{n}$ of rank $k.$ Then, $q$ represents a trivial
projection in $M_{n}(C(X))$ and since $r$ is trivial on $Y,$ from
Serre-Swan correspondence there exists $s_{0}\in M_{n}(C(Y))$ such
that $s_{0}^{\text{\textasteriskcentered}}s_{0}=q\restriction_{Y}$
and $r=s_{0}s_{0}^{\text{\textasteriskcentered}}.$ Let $p$ be the
unit of $M_{n}(C(X))$. Then, as $s_{0}s_{0}^{\text{*}}\leq p\restriction_{Y}$
and $\text{rank}(q)\leq\text{rank}(p)-\left\lfloor \frac{d}{2}\right\rfloor $,
by Theorem 2.5 $s_{0}$ extends to $s\in M_{n}(C(X))$ such that,
$s^{\text{\textasteriskcentered}}s=q$. Thus, $\tilde{r}=ss^{\text{\textasteriskcentered}}$
is a trivial projection on $X$ with $\tilde{r}\restriction_{Y}=s_{0}s_{0}^{\text{\textasteriskcentered}}=r.$\end{proof}

\begin{rem*}
\noindent In the hypothesis of 2.6 its assumed that $q$ is defined
on $X.$ However, 2.7 tell us that the conclusion of 2.6 is valid
even when $q$ is defined only on $Y$. In section 3, we will use
this observation without further mention. 
\end{rem*}

\noindent The final result of this section will not be required in
proving Theorem 3.5 but we will need this to prove Lemma 4.2, which
is essential to our proof of the main Theorem of section 4. 

\noindent For $a\in M_{n}(C(X))$, let $spec(a)=\{\lambda\in\mathbb{C}:a-\lambda1_{n}\text{ is not invertible}\}$.
For $a\in M_{n}(C(X))_{+}$, let $\Gamma_{a}:X\to\mathbb{R}^{n}$
be the map given by 
\[
\Gamma_{a}(x)=(\lambda_{1}(x),\,\lambda_{2}(x)........,\,\lambda_{n}(x)),
\]
where $\lambda_{1}(x)\leq\lambda_{2}(x)\leq.....\leq\lambda_{n}(x)$
are the eigenvalues of $a(x)$. Proposition \ref{final2} is a straightforward
consequence of the continuity of $\Gamma_{a}$ on $X$.

\begin{prop}\label{final2}
\noindent Let $a\in M_{n}(C(X))_{+}$ and $\eta\geq0$. The map $x\mapsto\text{rank}[\chi_{(\eta,||a||]}(a(x))]$
where $\chi_{(\eta,||a||]}$ denotes the characteristic map on $(\eta,||a||]$,
is lower semi-continuous.\end{prop}

\begin{proof}
\noindent It is clear that $\Gamma_{a}$ is continuous on $X$ for
any $a\in M_{n}(C(X))_{+}$. Let $x\in X$ and suppose $\text{rank}[\chi_{(\eta,||a||]}(a(x))]=m$.
Then, there are exactly $m$ (with possible repetitions) eigenvalues
of $a(x),$ which are grater than $\eta.$ Moreover, as $\lambda_{i}(x)$
are in increasing order, $\lambda_{n-m+1}(x),\lambda_{n-m+2}(x),....\lambda_{n}(x)$
are exactly the eigenvalues of $a(x)$ which are greater than $\eta$.
Set $\epsilon=\lambda_{n-m+1}(x)-\eta>0$ and by continuity of $\Gamma_{a}$
choose a neighborhood $U_{x}$ of $x$ such that $||\Gamma_{a}(x)-\Gamma_{a}(y)||<\epsilon,\forall y\in U_x$. 

\noindent Then for all $1\leq i\leq n$, 
\[
|\lambda_{i}(x)-\lambda_{i}(y)|<\epsilon
\]
and therefore by the choice of $\epsilon$, 
\[
\forall y\in U_{x},\forall n-m+1\leq i\leq n,\,\lambda_{i}(y)>\eta.
\]
Thus, for each $y\in U_{x}$, 
\[
\text{rank}[\chi_{(\eta,||a||]}(a(y))]\geq m=\text{rank\ensuremath{[\chi_{(\eta,||a||]}(a(x))]}.}
\]

\end{proof}

\section{\noindent Proof of the main result}

\noindent Our first aim is to prove Lemma 3.2, which is the main technical
result. Next we prove Proposition 3.4 which extends \cite[Chap. 8, Theorem 7.2]{Huse}
to compact Hausdorff spaces. Theorem 3.5 follows immediately by combining
the two Lemmas. 

\noindent The following is a direct consequence of Theorem 2.3. 
\begin{lem}
\noindent Let $X$ be a compact Hausdorff space and \mbox{$a\in M_{n}(C(X))_{+}$}.
Then, there exists a continuous path $t\mapsto a_{t}$ in $M_{n}(C(X))_{+}$
connecting $a_{0}=a$ to $a_1$, where $a_1$ is well supported in the sense of Definition 2.2 and has the same rank values as that of $a$. The path is such that $\text{rank}(a_{t}(x))=\text{rank}(a(x)),\forall t\in (0,1),\forall x\in X$. \end{lem}
\begin{proof}
\noindent Applying Theorem 2.3, choose a well supported positive element
$b\leq a$ such that $b$ has the same rank values as that of $a$,
let $a_{t}=(1-t)a+tb$. We only have to verify that $\text{rank }(a_{t}(x))=\text{rank }(a(x))$
for every $t\in(0,1)$ and $x\in X.$ But this is immediate.

\noindent Since $b\leq a$, if $0<t<1,$ 
\[
(1-t)a\leq a_{t}\leq(1-t)a+ta=a.
\]
Therefore, $\text{rank}(a_{t}(x))=\text{rank}(a(x))$, $\forall t\in(0,1),\forall x\in X$. \end{proof}
\begin{lem}
\noindent Let $X$ be a compact Hausdorff space with dim $X<\infty$.
Suppose $n,k,l\in\mathbb{N}$ are such that $k\leq n$ and $k-l\geq\left\lfloor \frac{dim\, X}{2}\right\rfloor $
and let $a\in C(X,S(n,k,l))$. Then, there is a continuous path $h$
: $[0,1]\to C(X,S(n,k,l))$ such that $h(0)=a$ and $h(1)$ is a trivial
projection of rank $l$.\end{lem}
\begin{proof}
\noindent Let $X,n,k,l$ and $a$ be as given in the hypothesis. By
Lemma 3.1 we can clearly assume that $a$ is well supported. 

\noindent Let the rank values of $a$ be $n_{1}<n_{2}<.....<n_{L}$
and let $E_{1},E_{2}\ldots,E_{L}$ and $p_{1},p_{2},\ldots,p_{L}$
be as in Definition 2.2. For convenience we will write $F_{i}=\overline{E}_{i}$
and $d=dim\, X$. 

We first consider the case $n_{L}\leq\lfloor\frac{d}{2}\rfloor$.
Then, choose $p\in M_{n}(C(X))$ to be any trivial projection of rank
$l$ and let 
\[
h(t)=(1-t)a+tp.
\]
Now for each $t\in[0,1],x\in X,$
\begin{eqnarray*}
\mbox{rank}\,[h(t)(x)] & \leq & \mbox{rank}\, a(t)(x)+\mbox{rank}\, p\\
 & \leq & n_{L}+l\\
 & \leq & \lfloor\frac{d}{2}\rfloor+l\\
 & \leq & k,
\end{eqnarray*}

and clearly $\mbox{rank}\,[h(t)(x)]\geq l$. Thus, we get the required
path.

Now let us assume $n_{L}>\lfloor\frac{d}{2}\rfloor.$ 

Fix $r$ such that $n_{r}>\lfloor\frac{d}{2}\rfloor$and $n_{r-1}\leq\lfloor\frac{d}{2}\rfloor,$
where we allow the possibility $r=1$ and set $n_{0}=0,\, F_{0}=\emptyset$.

In what proceeds, we will construct a trivial projection $R\in M_{n}(C(X))$
of rank $l$ such that, 
\begin{eqnarray*}
\mbox{rank}\,(R+a)(x) & \leq & k,\forall x\in X.
\end{eqnarray*}

Once we have such $R,$ we define $h:[0,1]\to M_{n}(C(X))$ by,

\noindent 
\[
h(t)=(1-t)a+tR,\forall t\in[0,1].
\]

Then its immediate that this path satisfies the said rank constrains
(i.e. remains in side $C(S(n,k,l))$. 

We focus on constructing $R.$ To this end, we first define a trivial
projection $q_{L}\in M_{n}(C(\underset{r\leq j\leq L}{\bigcup}F_{j}))$
such that 

\begin{eqnarray*}
\mbox{rank}\, q_{L} & = & n_{L}-\lfloor\frac{d}{2}\rfloor
\end{eqnarray*}
and
\[
\mbox{rank }(a+q_{L})(x)\leq n_{L},\forall x\in\underset{r\leq j\leq L}{\bigcup}F_{j}.
\]

We follow an inductive argument to define $q_{L}$.

Since $F_{r}$ is compact Hausdorff with $dim\: F_{r}\leq d$ and
$p_{r}\in M_{n}(C(F_{r}))$ is a projection of rank $n_{r}>\lfloor\frac{d}{2}\rfloor$,
using Theorem 2 .1 (1) and Serre-Swan correspondence we find a trivial
projection $q_{r}\in M_{n}(C(F_{r}))$ such that,
\begin{eqnarray*}
\mbox{rank}\, q_{r} & = & n_{r}-\lfloor\frac{d}{2}\rfloor\text{ }
\end{eqnarray*}
and $q_{r}\leq p_{r}.$

By the requirements for well supportedness of $a$, each $p_{i}\in M_{n}(C(F_{i}))$
is of constant rank $n_{i}$ and whenever $r\leq i\leq j$ with $F_{i}\cap F_{j}\neq\emptyset$,
\begin{equation}
p_{i}(x)\leq p_{j}(x),\,\forall x\in F_{i}\cap F_{j}.\label{eq:3.1}
\end{equation}

Also for all $j\geq r,$
\begin{eqnarray*}
\mbox{rank}\, p_{j}-\mbox{rank}\, q_{r} & \geq & n_{r}-\mbox{rank}\, q_{1}\\
 & \geq & \lfloor\frac{d}{2}\rfloor.
\end{eqnarray*}

Hence, we apply 2.6 with $X=\underset{r\leq j\leq L}{\bigcup}F_{j},$
$Y=F_{1}$, $q=q_{r}$ (by the remark following 2.7, $q$ in 2.6 need
not be defined on $X$ a priori) to extend $q_{r}$ to a trivial projection
in $M_{n}(C(\underset{r\leq j\leq L}{\bigcup}F_{j})))$ - which is
again called $q_{r}$, such that whenever $r\leq j,$
\begin{eqnarray*}
q_{r}(x) & \leq & p_{j}(x),\forall x\in F_{j}.
\end{eqnarray*}

Then, for each $r\leq j\leq L,$ 
\[
\mbox{rank}\,(q_{r}+a)(x)\leq n_{j},\forall x\in F_{j}.
\]

If $r=L$ then we are done (defining $q_{L}$).

Thus, let us assume $r<L.$

Suppose that for some $r\leq t<L$ we have defined a trivial projection
$q_{t}\in M_{n}(C(\underset{r\leq j\leq L}{\bigcup}F_{j}))$ such
that the following hold,

\noindent 
\begin{eqnarray}
\mbox{rank}\, q_{t} & = & n_{t}-\lfloor\frac{d}{2}\rfloor,\label{eq:3.2}\\
q_{t}(x) & \leq & p_{j}(x),\forall x\in F_{j},\forall t\leq j\leq L,\label{eq:3.3-1}\\
\mbox{rank}\,(q_{t}+p_{j}) & \leq & n_{t},\forall r\leq j\leq t.\label{eq:3.4-1}
\end{eqnarray}

\noindent Now whenever $t+1\leq j\leq L$, $(p_{j}-q_{t})\restriction_{F_{t+1}}\in M_{n}(C(F_{i+1}))$
is a projection constant rank $(n_{j}-n_{t})+\lfloor\frac{d}{2}\rfloor$. 

\noindent Thus, since $dim\, F_{t+1}\leq d$, by applying 2.1 we choose
a trivial projection $q_{t,t+1}\in M_{n}(C(F_{t+1}))$ such that,
\begin{eqnarray*}
\mbox{rank}\, q_{t,t+1} & = & n_{t+1}-n_{t}
\end{eqnarray*}

and $q_{t,t+1}\leq p_{t+1}-q_{t}.$

\noindent Moreover, by applying 2.6 with $X=\underset{t+1\leq j\leq L}{\bigcup}F_{j}$,
$Y=F_{t+1}$ and $q=q_{t,t+1}$ we extend $q_{t,t+1}$ to a trivial
projection in $M_{n}(C(\underset{t+1\leq j\leq L}{\bigcup}F_{j}))$
(which we again name $q_{t,t+1}$) such that whenever $j\geq t+1$,
\begin{equation}
q_{t,t+1}(x)\leq p_{j}(x)-q_{t}(x),\forall x\in F_{j}.\label{eq:3.5-1}
\end{equation}

\noindent Set $q_{t+1}=q_{t}+q_{t,t+1}$. 

\noindent Then, since $q_{t},\, q_{t,t+1}$ are orthogonal trivial
projections, $q_{t+1}$ is a trivial projection in $M_{n}(C(\underset{t+1\leq j\leq L}{\bigcup}F_{j})).$ 

\noindent Moreover,

by \ref{eq:3.2},
\begin{eqnarray*}
\mbox{rank}\, q_{t+1} & = & (n_{t}-\lfloor\frac{d}{2}\rfloor)+(n_{t+1}-n_{t})\\
 & = & n_{t+1}-\lfloor\frac{d}{2}\rfloor
\end{eqnarray*}

and whenever $j\geq t+1,$ $\forall x\in F_{j}$ (by \ref{eq:3.3-1}
and \ref{eq:3.5-1} ), 
\begin{eqnarray*}
q_{t+1}(x) & = & [q_{t}(x)+q_{t,t+1}(x)]\\
 & \leq & p_{j}(x)
\end{eqnarray*}
 and finally for each $r\leq j\leq t+1$ (by \ref{eq:3.4-1}),

\begin{eqnarray*}
\mbox{rank}\,(q_{t+1}+p_{j}) & \leq & \mbox{rank}\, q_{t,t+1}+\mbox{rank}\,(q_{t}+p_{j})\\
 & \leq & (n_{t+1}-n_{t})+n_{t}\\
 & = & n_{t+1}.
\end{eqnarray*}

By proceeding in this manner we construct a trivial projection $q_{L}\in M_{n}(C(\underset{r\leq j\leq L}{\bigcup}F_{j}))$
of rank $n_{L}-\lfloor\frac{d}{2}\rfloor$ such that, 
\begin{eqnarray*}
\mbox{rank}\,(q_{L}+p_{j})(x) & \leq & n_{L},\forall x\in F_{j},\forall r\leq j\leq L
\end{eqnarray*}

and

\[
q_{L}(x)\leq p_{L}(x),\forall x\in F_{L}.
\]

\noindent Choose $R_{1}\in M_{n}(C(X))$ to be any trivial projection
(of rank $n_{L}-\lfloor\frac{d}{2}\rfloor$) which extends $q_{L}$.
Note that such $R_{1}$ exists by Corollary 2.7. 

By the choice of $r,$ whenever $j<r,$ $\forall x\in F_{j},$
\begin{eqnarray*}
\mbox{rank}\,(R_{1}+a)(x) & \leq & (n_{L}-\lfloor\frac{d}{2}\rfloor)+\lfloor\frac{d}{2}\rfloor\\
 & \leq & n_{L}.
\end{eqnarray*}

Thus, since $R_{1}\restriction_{F_{j}}=q_{L}\restriction_{F_{j}}$
whenever $j\geq r$ we conclude that ,
\[
\mbox{rank}\,(R_{1}+a)(x)\leq n_{L},\forall x\in X.
\]

If $n_{L}=k,$ then 
\begin{eqnarray*}
\mbox{rank}\, R_{1} & = & k-\lfloor\frac{d}{2}\rfloor\\
 & \geq & l
\end{eqnarray*}

and we choose $R$ to be any trivial sub-projection of rank $l$ .

Hence we are left with the case $k>n_{L}$.

Then,

\begin{eqnarray*}
\mbox{rank}\,(1_{n}-R_{1}) & = & n-(n_{L}-\lfloor\frac{d}{2}\rfloor)\\
 & \geq & (k-n_{L})+\lfloor\frac{d}{2}\rfloor
\end{eqnarray*}

and we apply Theorem 2.1 (1) and Serre-Swan for one last time to choose
a trivial projection $R_{2}\in M_{n}(C(X))$ of rank $k-n_{L}$ with
$R_{2}\leq(1_{n}-R_{1}).$ 

Now $R_{1}+R_{2}$ is a trivial projection of rank $k-\lfloor\frac{d}{2}\rfloor$
and
\[
\mbox{rank}\,(R_{1}+R_{2}+a)(x)\leq k,\forall x\in X.
\]

To complete the proof we choose $R$ to be a trivial sub-projection
of $R_{1}+R_{2}$ of rank $l$.
\end{proof}
\noindent Recall the following theorem for locally trivial vector
bundles bundles over $CW$ -complexes.
\begin{thm}
\noindent \cite[Chp. 8, Theoreom 7.2]{Huse} Let $X$ be a CW-complex
and $n,l$ be non-negative integers. Then, the function that assigns
to each homotopy class $\left[f\right]:X\to G_{l}(\mathbb{C}^{n})$
the isomorphism class of the $k$-dimensional vector bundle $f^{*}(\gamma_{k}^{n})$
over $X$ is a bijection, if \mbox{$n\geq l+\left\lceil \frac{dim\, X}{2}\right\rceil $}. 
\end{thm}
\noindent Combining 3.3 with Lemma 3.2 proves Theorem 3.5 for all
$CW$-complex. The fact that Theorem 3.3 extends to the case of $X$
being Compact Hausdorff is probably well known. However, we could
not find a clear reference for this in the literature. Note that since
$G_{l}(\mathbb{C}^{\infty})$ (the $l$-dimensional Grassmannian over
$\mathbb{C}^{\infty}$) is the classifying space for $l$-dimensional
vector bundles over paracompact spaces, if one replaces $G_{l}(\mathbb{C}^{n})$
by $G_{l}(\mathbb{C}^{\infty})$ the conclusion of 3.3 holds even
for paracompact spaces. But the application we have in mind require
the target space to be $G_{l}(\mathbb{C}^{n})$. In Proposition 3.4,
based on the proof of \cite[Theorem 2.5]{Goodreal} we apply dimension
theory results \cite{ElSt,Nagami} and few $C^{*}$-algebraic techniques
to extend the conclusion of 3.3 to compact Hausdorff spaces.

\noindent If $A,B$ are $C^{*}$-algebras and $\phi:A\to B$ is a
$*$-homomorphism, then for any $n\in\mathbb{N}$ we have an induced
$*$-homomorphism from $M_{n}(A)$ to $M_{n}(B)$ given by $[a_{ij}]\mapsto[\phi(a_{ij})].$
We will use $\phi$ to denote this $*$-homomorphism as well.

\noindent For two projections $p,q\in M_{n}(C(X))$, we write $p\sim_{h\,}q$
if there is a projection valued continuous path in $M_{n}(C(X))$
which connects $p$ and $q$. 

\begin{prop}
\noindent Let $X$ be a compact Hausdorff space with dim $X<\infty$
and suppose $n,k\in\mathbb{N}$ with $n-k\geq\left\lceil \frac{dim\, X}{2}\right\rceil $.
Then the isomorphism classes of $k$-dimensional locally trivial complex
vector bundles over $X$ are in bijective correspondence with the
homotopy classes of maps $p\colon X\to P_{k}(\mathbb{C}^{n})$, where
$P_{k}(\mathbb{C}^{n})$ stands for rank $k$ projections in $M_{n}(\mathbb{C})$. \end{prop}

\begin{proof}
\noindent Let $[X,P_{k}(\mathbb{C}^{n})]$ stand for the homotopy
classes of maps in $P_{k}(M_{n}(C(X)))$ and let $Vect_{k}(X)$ denote
the set of all isomorphic classes of locally trivial $k$-dimensional
vector bundles over $X.$ 

\noindent Define $\psi:[X,P_{k}(\mathbb{C}^{n})]\to Vect_{k}(X)$
by $\psi([p])=[\epsilon_{p}], \forall p\in P_{k}(M_{n}(C(X)))$, where $\epsilon_{p}$ is defined as in section 2. Since the homotopy
equivalence of projections implies Murray-von Neumann equivalence
(see \cite[Prop. 2.2.7]{LLR}), map $\psi$ is well defined by the
discussion in section 2. Moreover, the discussion following Theorem
2.1 shows that $\psi$ is surjective.

\noindent To complete the proof of 3.4, we only have to show that
if the vector bundles associated with two projections in $P_{k}(M_{n}(C(X)))$
are isomorphic then the two projections are homotopic in $P_{k}(M_{n}(C(X)))$.

\noindent Let us first assume that $X$ is a compact metric space. 

\noindent Then by \cite[Chap. 27, Theorem 8.]{Nagami} or \cite[Chap. X, Sec. 10]{ElSt},
$X$ is homeomorphic to an inverse limit of finite simplicial complexes
$X_{\alpha}$, with $dim\, X_{\alpha}\leq dim\, X$ for each $\alpha$. 

\noindent Let $\psi_{\alpha}:X\to X_{\alpha}$ be the canonical induced
maps. We have the induced homomorphisms, 
\[
\psi_{\alpha}^{T}:C(X_{\alpha})\to C(X)
\]
given by $\psi_{\alpha}^{T}(f)=f\circ\psi_{\alpha}$. 

\noindent Moreover, by the inverse limit structure
\begin{equation}
\overline{\underset{\alpha}{\bigcup}\psi_{\alpha}^{T}(C(X_{\alpha}))}=C(X),\label{eq:3.8}
\end{equation}
i.e. $\underset{\alpha}{\bigcup}\psi_{\alpha}^{T}(C(X_{\alpha}))$
is a $||\text{ }||_{sup}$ dense $*$-subalgebra of $C(X).$ Note
that $\text{if }\alpha<\beta$ then, 
\begin{equation}
\psi_{\alpha}^{T}(C(X_{\alpha}))\subset\psi_{\beta}^{T}(C(X_{\beta}))\label{eq:3.9}
\end{equation}
and hence $\underset{\alpha}{\bigcup}\psi_{\alpha}^{T}(C(X_{\alpha}))$
is indeed a $*$-subalgebra. 

\noindent Suppose $p,q\in M_{n}(C(X))$ are projections of constant
rank $k$ such that the corresponding vector bundles are isomorphic.
This means, $p=v^{*}v,q=vv^{*}$ for some $v\in M_{n}(C(X))$. 

\noindent Fix $0<\epsilon<1/3.$ By (\ref{eq:3.8}), there is some
$\alpha$ and $\tilde{v}_{\alpha}\in M_{n}(C(X_{\alpha})))_{1}$ such
that, 
\[
||\psi_{\alpha}^{T}(\tilde{v}_{\alpha})-v||<\epsilon/2.
\]
Write $Y_{\alpha}=\psi_{\alpha}(X)$ and $v_{\alpha}=\tilde{v}_{\alpha}\restriction_{Y_{\alpha}}$,
$a_{\alpha}=v_{\alpha}^{*}v_{\alpha}$, $b_{\alpha}=v_{\alpha}v_{\alpha}^{*}$. 

\noindent Now $a_{\alpha}\in M_{n}(C(Y_{\alpha}))$ and moreover by
the choice of $v_{\alpha}$ it follows that, 
\[
spec(a_{\alpha})\subset[0,\epsilon)\cup(1-\epsilon,1].
\]
Similarly, 
\[
spec(b)\subset[0,\epsilon)\cup(1-\epsilon,1].
\]

\noindent Let $f\colon[0,1]\to[0,1]$ be the continuous function which
vanishes on $[0,\epsilon],$ is equal to 1 on $[1-\epsilon,1]$ and
is linear on $(\epsilon,1-\epsilon)$. Then $f(a_{\alpha})$, $f(b_{\alpha})$
are projections in $M_{n}(C(Y_{\alpha}))$ with 
\[
||a_{\alpha}-f(a_{\alpha})||<\epsilon,||b_{\alpha}-f(b_{\alpha})||<\epsilon.
\]
Furthermore, as $a,b$ are Murray-von Neumann equivalent, by Lemma
3.3 of \cite{DT} there is some $s_{0}\in M_{n}(C(Y_{\alpha}))$ such
that, 
\[
f(a_{\alpha})=s_{0}^{*}s_{0},\, f(b_{\alpha})=s_{0}s_{0}^{*}.
\]
Since $Y_{\alpha}$ is a closed in $X_{\alpha},$ we may choose some
open neighborhood $U$ of $Y_{\alpha}$ so that $f(a_{\alpha}),f(b_{\alpha})$
extends to projections in $M_{n}(C(U))$. Let $p_{\alpha},q_{\alpha}$
denote these extensions respectively. Moreover, since $f(a_{\alpha})\sim f(b_{\alpha})$,
we may choose the extensions in such a way that $p_{\alpha}\sim q_{\alpha}$.

\noindent As $X_{\alpha}$ is a finite simplicial complex, after a
finite simplicial refinement of $X_{\alpha}$ via barycentric subdivisions,
choose a sub complex $Z$ of $X_{\alpha}$ with $Y\subset Z\subset U$.
For convenience let us denote the restrictions of $p_{\alpha},q_{\alpha}$
to $Z$ by $p_{\alpha},q_{\alpha}.$ 

\noindent Note that $p_{\alpha},q_{\alpha}$ generate isomorphic vector
bundles over $Z,$ each of rank $k.$ Hence, form Theorem 3.3 and
the identification of $G_{k}(\mathbb{C}^{n})$ with $P_{k}(\mathbb{C}^{n})$,
there is a continuous path,

\noindent 
\[
t\mapsto h_{\alpha}(t)\in P_{k}(M_{n}(C(Z))),
\]
such that $h_{\alpha}(0)=p_{\alpha}$, $h(1)=q_{\alpha}.$ 

\noindent This gives a path $t\mapsto h(t)\in P_{k}(M_{n}(C(X))),$
given by $h(t)(x)=h_{\alpha}(t)(\psi_{\alpha}(x))$. 

\noindent Note that for all $x\in X,$

\noindent 
\begin{eqnarray*}
||p(x)-h(0)(x)|| & = & ||p(x)-f(a_{\alpha})(\psi_{\alpha}(x))||\\
 & \leq & ||v^{*}v(x)-(\tilde{v}_{\alpha}^{*}\tilde{v}_{\alpha})(\psi_{\alpha}(x))||+||(\tilde{v}_{\alpha}^{*}\tilde{v}_{\alpha})(\psi_{\alpha}(x))-f(a_{\alpha})(\psi_{\alpha}(x))||\\
 & \leq & 2\epsilon+||a_{\alpha}(\psi_{\alpha}(x))-f(a_{\alpha})(\psi_{\alpha}(x))||\\
 & \leq & 2\epsilon+\epsilon
\end{eqnarray*}

\noindent Thus, $||p-h(0)||<1$ and similarly $||q-h(1)||<1$. 

\noindent Therefore, from \cite[Proposition 2.2.4]{LLR}
\[
p\sim_{h\,}h(0)\sim_{h\,}h(1)\sim_{h\,}q.
\]
This completes the proof for compact metric spaces.

\noindent The proof of the Proposition in the case of $X$ being an
arbitrary compact Hausdorff space follows almost identically. In this
case $X$ is the inverse limit of compact metric spaces $X_{\lambda}$,
with $dim\, X_{\lambda}\leq dim\, X$ for each $\lambda.$ By the
preceding the result holds for each $X_{\lambda}$, and now we can
argue as in the preceding case. \end{proof}
\begin{thm}
\noindent Let $X$ be a compact Hausdorff space $X$ with $\left\lfloor \frac{dim\, X}{2}\right\rfloor \leq k-l.$
There is only one homotopy class of functions $f:X\to S(n,k,l),$
i.e. the function space $C(X,S(n,k,l))$ is path connected.\end{thm}
\begin{proof}
\noindent Observe that the case $n=k$ is straightforward. For any
$a\in C(X,S(n,k,l))$ we have the linear path $t\mapsto(1-t)a+1_{n}$
connecting $a$ to $1_{n}$. 

\noindent So we assume $n>k$. 

\noindent Let $a,b\in C(X,S(n,k,l)).$ As $dim\, X\leq k-l,$ by applying
Lemma 3.2 choose trivial projections $p,q$ of rank $l$ such that
there are paths inside $C(X,S(n,k,l))$ connecting $a$ to $p$ and
$b$ to $q$. Since $n>k$ and $k-l\geq\left\lfloor \frac{dim\, X}{2}\right\rfloor $,
we get $n-l\geq\lceil\frac{dim\, X}{2}\rceil$. Therefore, as $p,q$
are both trivial, by Proposition 3.4 there is a path inside $C(X,S(n,k,l))$
connecting $p$ and $q$. Thus, there is a path between $a$ and $b$
in $C(X,S(n,k,l))$.\end{proof}
\begin{cor}
\noindent For every $r\leq2(k-l)+1$, \textup{$\pi_{r}(S(n,k,l))=0$.}\end{cor}
\begin{proof}
\noindent Follows directly from Theorem 3.5 as $dim\, S^{r}=r$.
\end{proof}
\noindent Even though we derived the above result as a Corollary to
Theorem 3.5, from a topological view point it might seem more natural
(and technically easier) to first prove the conclusion of 3.6 independently
and then apply techniques from homotopy theory and dimension theory
to prove Theorem 3.5. This is indeed possible and applies in a slightly
wider scope. In section 4, we show that for compact Hausdorff spaces
of finite covering dimension the path connectedness of $C(X,S(n,k,l))$
depends solely on homotopy groups $\pi_{r}(S(n,k,l))$, $r\leq dim\, X$.
We chose not to follow this method for the proof of 3.5, due to two
reasons. Firstly, with the techniques we used, the proof of Lemma
3.2 would not be any simpler if we assumed $X$ to be a sphere
instead of a general compact Hausdorff spaces. Secondly, if we used
such an argument (i.e proving 3.6 independently and then showing 3.5)
it would have avoided the use of Proposition 3.4, but we thought 3.4
could be of independent interest for some readers.

\section{\noindent Homotopy groups of $S(n,k,l)$ and Path Connectedness of
$C(X,S(n,k,l)).$}

\noindent In Theorem 4.6, combining well known homotopy theory techniques
and classical $C^{*}$-algebraic results we prove that for a fixed
integer $d$ if $\pi_{r}(S(n,k,l))=0, \forall r\leq d$, then $C(X,S(n,k,l))$ is path
connected for every compact Hausdorff space $X$ with $dim\, X\leq d$.
Note that in this section we do not assume $k-l\geq\left\lfloor \frac{dim\, X}{2}\right\rfloor .$ 

\noindent The proof of Lemma 4.5 is of the same flavor as that of
Proposition 3.4. We need few more technical results. 

\noindent The following is well known and was used in the proof of
3.4 as well. We state it here for convenience.
\begin{prop}
\noindent Let $X$ be a finite simplicial complex. Let $Y\subset X$
be closed and $U$ be a open neighborhood of $Y$ in $X$. Then, after
a finite simplicial refinement of $X$, there is a sub-complex $Z$
of $X$ such that $Y\subset Z\subset U$.
\end{prop}
\noindent Lemma 4.2 is a simpler version of \cite[Lemma 2.1]{tm1}.
For the sake of completeness, we provide a proof.
\begin{lem}
\noindent \cite[Lemma 2.1]{tm1} Let $X$ be a compact Hausdorff space
and suppose that $a\in M_{n}(C(X))_{+}$. Let $l\in\mathbb{N}$ be
such that $\text{rank}(a(x))\geq l$, $\forall x\in X$. Then, there
is some $\eta>0$ such that for each $x\in X$, the spectral projection
$\chi_{(\eta,\infty)}(a(x))$ has rank at least $l$.\end{lem}
\begin{proof}
\noindent For each $x\in X$, let $\eta_{x}=\frac{1}{2}\text{min\{\ensuremath{\lambda\in\textrm{spec}} \ensuremath{a(x)}:\ensuremath{\lambda>0}\}}$.
Note that since $l>0,$ $\eta_{x}$ exists for each $x\in X$. Then,
\[
\text{rank}(a(x))=\text{rank}[\chi_{(\eta_{x},\infty)}(a(x))],\forall x\in X.
\]
By Proposition 2.8, for each $x\in X$, the map $y\mapsto\mbox{\ensuremath{\text{rank[\ensuremath{\chi}}_{(\eta_{x},\infty)}(a(y))]}}$
is lower semi-continuous. So, for each $x\in X$, there exists an
open neighborhood $U_{x}$ of $x$ such that,

\noindent 
\[
\text{\text{rank}\ensuremath{[\chi_{(\eta_{x},\infty)}(a(y))]\geq\text{rank}[\ensuremath{\chi_{(\eta_{x},\infty)}(a(x))]},\forall y\in U_{x}.}}
\]

\noindent By compactness of $X$, choose some finite set of points
$\left\{ x_{1},x_{2},...x_{L}\right\} $ such that $X=\underset{1\leq i\leq L}{\bigcup}U_{x_{i}}$. 

\noindent By setting $\eta=\min_{1\leq i\leq L}(\eta_{x_{i}})>0$,
for every $y\in U_{x_{i}}$ we get, 
\begin{eqnarray*}
\text{\text{rank}[\ensuremath{\chi_{(\eta,\infty)}(a(y))]}} & \ge & \text{rank}[\ensuremath{\chi_{(\eta_{x_{i}},\infty)}(a(y))]}\\
 & \geq & \text{rank}[\chi_{(\eta_{x_{i}},\infty)}(a(x))]\\
 & = & \text{rank}(a(x))\\
 & \geq & l.
\end{eqnarray*}

\noindent Since $X=\underset{1\leq i\leq L}{\bigcup}U_{x_{i}}$, this
completes the proof.
\end{proof}
\noindent Let $X=\underleftarrow{\lim}\, X_{\alpha}$, where $(X_{\alpha},\psi_{\alpha\beta})$
is a inverse system of compact Hausdorff spaces and $\psi_{\alpha}:X\to X_{\alpha}$
be the natural maps. In the proof of Lemma 4.5, given $a\in C(X,S(n,k,l))\subset M_{n}(C(X))_{+}$
we need to approximate $a$ in norm, by elements in $\psi_{\alpha}^{T}(C(X_{\alpha},S(n,k,l)))$.
Here, $\psi_{\alpha}^{T}$ is defined as in the proof of Proposition
3.4. To achieve this, we will use the previous Lemma with two other
results from $C^{*}$-algebra theory. First we prove Lemma 4.3, which
follows as a special case of a well known fact in $C^{*}$-algebras.
For the sake of completeness we include the proof for the special
case.
\begin{lem}
\noindent Let $X=\underleftarrow{\lim}\, X_{\alpha}$, for a inverse
system of compact Hausdorff spaces $(X_{\alpha},\psi_{\alpha\beta})$.
Let $\psi_{\alpha}:X\to X_{\alpha}$ be the natural maps and $\epsilon>0$.
Given $a\in M_{n}(C(X))_{+},$ there is some index $\alpha$ and some
$b\in M_{n}(C(X_{\alpha}))_{+}$ such that,
\[
||a-\psi_{\alpha}^{T}(b)||<\epsilon
\]
where $\psi_{\alpha}^{T}(b)=b\circ\psi_{\alpha}$.\end{lem}
\begin{proof}
\noindent We may assume $||a||=1$ and $\epsilon<1$. Since $a$ is
positive, by the functional calculus of $a$ there is some $c\in M_{n}(C(X))_{+}$
such that $c^{2}=a$. As in the proof of 3.4, pick some $\alpha$
and $d\in M_{n}(C(X_{\alpha}))$ such that,
\begin{equation}
||c-\psi_{\alpha}^{T}(d)||<\frac{\epsilon}{3}.\label{eq:4.1}
\end{equation}
Since $c$ is positive, $c^{*}=c$ where $c^{*}$ is the conjugate
transpose of $c$. 

\noindent Then,
\begin{eqnarray}
||c-\psi_{\alpha}^{T}(d^{*}))|| & = & ||(c^{*}-\psi_{\alpha}^{T}(d^{*}))||\nonumber \\
 & = & ||(c-\psi_{\alpha}^{T}(d))^{*}||\nonumber \\
 & = & ||c-\psi_{\alpha}^{T}(d)||\nonumber \\
 & < & \frac{\epsilon}{3}\label{eq:4.2}
\end{eqnarray}
 Put $b=d^{*}d$ then $b\in M_{n}(C(X_{\alpha}))_{+}$ and moreover,

\noindent 
\begin{eqnarray*}
||a-\psi_{\alpha}^{T}(b))|| & = & ||c^{2}-\psi_{\alpha}^{T}(d^{*}d)||\\
 & = & ||c^{2}-\psi_{\alpha}^{T}(d^{*})\psi_{\alpha}^{T}(d)||\\
 & \leq & ||c^{2}-c\cdot\psi_{\alpha}^{T}(d)||+||c\cdot\psi_{\alpha}^{T}(d)-\psi_{\alpha}^{T}(d^{*})\psi_{\alpha}^{T}(d)||\\
 & \leq & ||c||\cdot||c-\psi_{\alpha}^{T}(d)||+||c-\psi_{\alpha}^{T}(d^{*})||\cdot||\psi_{\alpha}^{T}(d)||\\
 & < & \epsilon
\end{eqnarray*}

\noindent The last inequality follows from (\ref{eq:4.1}), (\ref{eq:4.2})
and the fact that $||\psi_{\alpha}^{T}(d)||\leq2$.
\end{proof}
\noindent For $\epsilon\geq0,$ let $f_{\epsilon}:[0,1]\to[0,1]$
be defined by, 
\[
f_{\epsilon}(t)=\max\{\epsilon,t\},\forall t\in[0,1].
\]

\noindent For any $a\in M_{n}(C(X))_{+}$, let $(a-\epsilon)_{+}$
denote the element $f_{\epsilon}(a)\in M_{n}(C(X))_{+}$ given by
the functional calculus of $a.$ Recall that by support projection
of $d\in M_{n}(C(X))$ we mean the function (not necessarily continuous)
which maps $x$ to the orthogonal projection of $\mathbb{C}^{n}$
onto $d(x)(\mathbb{C}^{n})$. Note that the support projection of
$f_{\epsilon}(a)$ is the spectral projection $\chi_{(\epsilon,1]}(a).$
To prove 4.5, we need the following proposition.
\begin{prop}
\noindent \cite[Proposition 2.2]{Rod1} Let $X$ be compact Hausdorff
and $a,b\in M_{n}(C(X))_{+}.$ Let $\epsilon>0$ and suppose that
$||b-a||<\epsilon.$ Then there exists $c\in M_{n}(C(X))$ such that,

\noindent 
\[
(a-\epsilon)_{+}=c^{*}bc.
\]

\end{prop}
\noindent We now prove Lemma 4.5.
\begin{lem}
\noindent Suppose for each finite simplicial complex $Z$ with $dim\, Z\leq d$,
the function space $C(Z,S(n,k,l))$ is path connected. Then, for every
compact Hausdorff space $X$ of covering dimension $d,$ space $C(X,S(n,k,l))$
is path connected.\end{lem}
\begin{proof}
\noindent Like in the proof of 3.4, we first prove the result for
the case of $X$ being a compact metric space. In this case $X=\underleftarrow{\lim}\, X_{\alpha}$,
where $(X_{\alpha},\psi_{\alpha\beta})$ is a inverse system finite
simplicial complexes with $dim\, X_{\alpha}\leq d$. Let $\psi_{\alpha}:X\to X_{\alpha}$
be the natural maps. 

\noindent Fixed $a\in C(X,S(n,k,l))$, our first goal is to show that
there is some index $\alpha$ and $c\in C(X_{\alpha},(M_{n})_{+})$
such that $\psi_{\alpha}^{T}(c)\in C(X,S(n,k,l))$ and there is a
path in $C(X,S(n,k,l))$ connecting $a$ to $\psi_{\alpha}^{T}(c)$. 

\noindent To do this, first note that we may assume $||a||=1$ and
use Lemma 4.2 to pick $\eta>0$ such that, 
\begin{equation}
\text{rank }[\chi_{(2\eta,1]}(a(x))]\geq l,\forall x\in X.\label{eq:4.3}
\end{equation}

\noindent By Lemma 4.3, pick $\alpha$ and $b\in M_{n}(C(X_{\alpha}))_{+}$
such that,
\begin{equation}
||a-\psi_{\alpha}^{T}(b)||<\eta.\label{eq:4.4}
\end{equation}

\noindent Let us write $a_{\alpha}=\psi_{\alpha}^{T}(b)$. Note that
from Proposition 4.4 and (\ref{eq:4.4}),
\[
(a_{\alpha}-\eta)_{+}=d^{*}ad,
\]

\noindent for some $d\in M_{n}(C(X)).$

\noindent Therefore, for every $x\in X,$
\begin{eqnarray}
\text{rank }[(a_{\alpha}-\eta)_{+}(x)] & \le & \text{rank }(a(x))\nonumber \\
 & \leq & k\label{eq:4.5}
\end{eqnarray}

\noindent From the functional calculus for $a_{\alpha}$ and \ref{eq:4.4},
\begin{eqnarray*}
||(a_{\alpha}-\eta)_{+}-a|| & \leq & ||(a_{\alpha}-\eta)_{+}-a_{\alpha}||+||a_{\alpha}-a||\\
 & < & \eta+\eta\text{ }\\
 & = & 2\eta.
\end{eqnarray*}

\noindent Therefore, by Proposition 4.4 it follows that,
\[
\text{rank }[(a-2\eta)_{+}(x)]\leq\text{rank }[(a_{\alpha}-\eta)_{+}(x)],\forall x\in X.
\]

\noindent Now, from (\ref{eq:4.3}) and the discussion preceding Proposition
4.4,
\begin{equation}
\text{rank }[(a_{\alpha}-\eta)_{+}(x)]\geq l,\forall x\in X.\label{eq:4.6}
\end{equation}

\noindent Put $c=\psi_{\alpha}^{T}((b-\eta)_{+})$. 

\noindent Then, 

\noindent 
\begin{eqnarray*}
c & = & (b-\eta)_{+}\circ\psi_{\alpha}\\
 & = & ((b\circ\psi_{\alpha})-\eta)_{+}\\
 & = & (a_{\alpha}-\eta)_{+}.
\end{eqnarray*}

\noindent Thus by (\ref{eq:4.5}) and (\ref{eq:4.6}), $c\in C(X,S(n,k,l))$.

\noindent To complete our first goal consider $h:[0,1]\to C(X,S(n,k,l))$
given by,
\[
h(t)=[((1-t)a+ta_{\alpha})-\eta]_{+}.
\]
 We have $h(0)=(a-\eta)_{+}$ and $h(1)=c$.

\noindent Note that, 
\[
||a-((1-t)a+ta_{\alpha})||=t||a-a_{\alpha}||<\eta,\forall t\in[0,1].
\]

\noindent Thus, by following the same type of argument we used to
show $c\in C(X,S(n,k,l)),$ it follows that 
\[
h(t)=[((1-t)a+ta_{\alpha})-\eta]_{+}\in C(X,S(n,k,l)),\forall t\in[0,1].
\]

\noindent The continuity of $h$ can be established using functional
calculus arguments, see \cite[Lemma 1.25]{LLR}. 

\noindent First goal is now achieved once when we notice that $a$
is homotopic to $(a-\eta)_{+}$ as maps in $C(X,S(n,k,l)).$ But this
is immediate. Indeed, one can use the linear path $t\mapsto(1-t)a+t(a-\eta)_{+}.$

\noindent Now suppose $a,b\in C(X,S(n,k,l)).$ We need to construct
a path in $C(X,S(n,k,l))$ joining $a$ to $b$. From the first part
w.l.o.g. we may assume that $a=\psi_{\alpha}^{T}(c)$, $b=\psi_{\beta}^{T}(d)$,
for some $\alpha=\beta$ and $c,d\in M_{n}(C(X_{\alpha}))_{+}$. 

\noindent Put $Y=\psi_{\alpha}(X)\subset X_{\alpha}$. Then as $X$
is compact so is $Y$. Moreover, $Y$ is closed as each $X_{\alpha}$
is Hausdorff. 

\noindent For each $y=\psi_{\alpha}(x),$
\begin{eqnarray*}
\text{rank }(c(y)) & = & \text{rank }(c(\psi_{\alpha}(x)))\\
 & = & \text{rank }(a(x))
\end{eqnarray*}

\noindent Hence, 
\[
l\leq\text{rank }(c(y))\leq k,\forall y\in Y.
\]

\noindent Similarly,
\[
l\leq\text{rank }(d(y))\leq k,\forall y\in Y.
\]

\noindent Therefore as $Y$ is closed, by Lemma 2.7 of \cite{tm1}
there is some open neighborhood $U$ of $Y$ in $X_{\alpha}$ and
$\tilde{c},\tilde{d}\in C(U,S(n,k,l))$ such that $\tilde{c},\tilde{d}$
are extensions of $c,d$ respectively. 

\noindent By Proposition 4.1, after a refinement of the simplicial
structure of $X_{\alpha}$, we have a finite sub complex $Z$ of $X_{\alpha}$
such that, $Y\subset Z\subset U$. Hence, we may view $\tilde{c},\tilde{d}\in C(Z,S(n,k,l))$.
Now, as $Z$ is a finite simplicial complex with $dim\, Z\leq d$,
by the hypothesis there is a path $\tilde{g}:[0,1]\to C(Z,S(n,k,l))$
such that $\tilde{g}(0)=\tilde{c}$ and $\tilde{g}(1)=\tilde{d}$. 

\noindent Define a path $g:[0,1]\to C(X,M_{n})$ by,

\noindent 
\[
g(t)(x)=\tilde{g}(t)(\psi_{\alpha}(x)).
\]

\noindent By definition of $\tilde{g}$ it is clear that $g(0)=\psi_{\alpha}^{T}(c)=a$
, $g(1)=\psi_{\alpha}^{T}(d)=b$ and moreover, $g(t)\in C(X,S(n,k,l))$. 

\noindent This proves the result in the case of $X$ being a compact
metric space such that $dim\, X\leq d$. To get the result for $X$
being compact Hausdorff, write $X=\underleftarrow{\lim}\, X_{\alpha}$
where now $X_{\alpha}$ is compact metric with $dim\, X_{\alpha}\leq d$.
Since $C(X_{\alpha},S(n,k,l))$ is path connected for each $\alpha$
by the first part of the proof, following a similar argument as before
give the result. (Note here that the argument is simpler than in the
first step. Since $\psi_{\alpha}(X)\subset X_{\alpha}$ is compact
metric for any $\alpha$, we do not require Lemma 2.7 of \cite{tm1}.\end{proof}
\begin{thm}
\noindent Let $X$ be compact Hausdorff with $dim\, X\leq d.$ If
$\pi_{r}(S(n,k,l))=0$ for each $r\leq d$ then, $C(X,S(n,k,l))$
is path connected.\end{thm}
\begin{proof}
\noindent From Lemma 4.5, it suffices to show that $C(K,S(n,k,l))$
is path connected for every finite simplicial complex $K$ with $dim\, K\leq d.$
We will use induction on the number of simplexes in the complex $K$
to show this. 

\noindent If $K$ consists of a single simplex then result is true
since $K$ is contractible. 

\noindent Suppose now that result is true for every simplicial complex
which contains $r$ number of simplexes.

\noindent To complete the inductive step, let $K=L\cup\{s\}$ where
$L$ is a sub complex of $K$ containing $r$ number of simplexes
and $s$ is a $n$-simplex for some $n\leq d.$ 

\noindent Observe that $\{K,L\}$ is a \textsl{NDR} pair in the sense
of \cite{Wh}. Since $S(n,k,l)$ is locally compact, $S(n,k,l)$ is
compactly generated. By \cite[Diagram 6.3]{Wh}, following sequence
is exact in the category of sets with base points,
\begin{equation}
[C(L,S(n,k,l))]\overset{i_{*}}{\,\longleftarrow\,[}C(K,S(n,k,l))]\overset{}{\overset{p_{*}}{\,\longleftarrow\,[}C((K/L),S(n,k,l))],}\label{eq:4.7}
\end{equation}

\noindent where $K/L$ stands for the quotient space. For a space $Z$, by $[C(Z,S(n,k,l))]$
we mean the set of all homotopy classes of maps in $C(Z,S(n,k,l))$
and as the base point we may choose any constant map $z\mapsto a$,
for a fixed $a\in S(n,k,l).$ The maps $i_{*}$ and $p_{*}$ are the
maps induced by the inclusion $i:L\to K$ and the quotient map $p:K\to K/L$. 

\noindent By the induction hypothesis $[C(L,S(n,k,l))]$ consists
of a single point. Since $(K/L)\cong S^{n}$ and $n\leq d$, by assumption
$[C((K/L),S(n,k,l))]$ is also a singleton. Thus, by exactness of
(\ref{eq:4.7}), $[C(K,S(n,k,l))]$ contains only one point, i.e.
$C(K,S(n,k,l))$ is path connected.\end{proof}


\begin{thebibliography}{10}
\bibitem[1]{At} M. F. Atiyah (notes by D. W. Anderson) \emph{$K$-Theory},
W.A. Benjamin, Inc. New york, New York 1967.

\bibitem[2]{BPT} N.P. Brown, F. Perera, A.S. Toms, The Cuntz semigroup,
the Elliott conjecture, and dimension functions on $C^{^{*}}$-algebras,
\emph{J. Reine Angew. Math} \textbf{621}, (2008), 191-211 

\bibitem[3]{BT} N.P. Brown, A.S. Toms, There applications of the
Cuntz semigroup, \emph{Int. Math Res. Not.}, (2007), doi: 10.1093/imrnirnm068

\bibitem[4]{DT} M. Dadarlat, A.S. Toms, Ranks of operators in simple
$C^{*}$-algebras, \emph{J.Funct.Anal.}\textbf{259}, (2010), 1209-1229.

\bibitem[5]{ElSt} S. Eilenberg and N. Steenrod, \emph{Foundations of Algebraic Topology},
Princeton University Press, Princeton, New Jersey, 1952.

\bibitem[6]{Goodreal}K. R. Goodearl, Riesz decomposition in inductive
limit $C^{*}$-algebras, \emph{Rocky Mountain J. Math.} \textbf{24},
(1994), 1405\textendash 1430.

\bibitem[7]{Huse}D. Husemoller, \emph{Fiber Bundles} (Third Ed.),
Springer-Verlag, New York, Inc., 1994.

\bibitem[8]{LLR}F. Larsen, N. J. Lausten and M. R{\o}rdam, \emph{An introduction to $K$-Theory for -$C^*$-algebras},
Cambridge University press, Cambridge, 2000.

\bibitem[9]{Nagami}K. Nagami, \emph{Dimension Theory}, Academic
Press, New York, 1970. 

\bibitem[10]{Phil}N. C. Phillips, Recursive subhomogeneous algebras,
\emph{Trans. Am. Math. Soc.} \textbf{359} (2007), no. 10, 4595-4623.

\bibitem[11]{Rod1}M. R{\o}rdam, On the structure of simple $C^{*}$-algebras
tensored with a UHF-algebra, II, \emph { J. Funct. Anal.}\textbf{107},
(1992), 255-269

\bibitem[12]{Swn}R. G. Swan, Vector bundles and projective modules,
\emph{Trans. Amer. Math. Soc.} \textbf{105} (1962), 264-277.

\bibitem[13]{tm1}A. Toms, K-theoretic rigidity and slow dimension
growth, \emph{Invent. Math.} \textbf{183} (2011), no. 2, 225\textendash 244.

\bibitem[14]{tm2}A. S. Toms, Stability in the Cuntz semigroup of
a commutative $C^{*}$-algebra, \emph{Proc. Lond. Math. Soc. (3)} \textbf{96}
(2008), no. 1, 1\textendash 25.

\bibitem[15]{tm3}A. S. Toms, Comparison theory and smooth minimal
$C^{*}$-dynamics, \emph{Comm. Math. Phys} \textbf{96} (2009), no.
2, 401\textendash 433.

\bibitem[16]{Wh}G. W. Whitehead, \emph{Elements of Homotopy Theory},
Springer-Verlag, New York, 1978. \end{thebibliography}
\end{document}